\documentclass[11pt,twoside,reqno]{amsart}

\usepackage[OT1]{fontenc}
\usepackage{type1cm}
\usepackage{amssymb}

\usepackage[left=2.3cm,top=3cm,right=2.3cm]{geometry}

\numberwithin{equation}{section}

\theoremstyle{plain}
\newtheorem{theorem}{Theorem}[section]
\newtheorem{thm}[theorem]{Theorem}
\newtheorem{corollary}[theorem]{Corollary}

\newtheorem{proposition}[theorem]{Proposition}

\newtheorem{lemma}[theorem]{Lemma}

\theoremstyle{remark}
\newtheorem{remark}[theorem]{Remark}

\newtheorem{example}[theorem]{Example}

\theoremstyle{definition}

\newcommand{\BB}{\mathcal{B}}
\newcommand{\CC}{\mathcal{C}}
\newcommand{\DD}{\mathcal{D}}

\newcommand{\QQ}{\mathcal{Q}}

\newcommand{\R}{\mathbb{R}}

\newcommand{\N}{\mathbb{N}}

\newcommand{\iii}{\mathtt{i}}
\newcommand{\jjj}{\mathtt{j}}

\newcommand{\eps}{\varepsilon}

\newcommand{\roo}{\varrho}

\DeclareMathOperator{\diam}{diam}

\DeclareMathOperator{\spt}{spt}

\DeclareMathOperator{\essinf}{ess\,inf}

\DeclareMathOperator{\udimloc}{\overline{dim}_{loc}}
\DeclareMathOperator{\ldimloc}{\underline{dim}_{loc}}

\DeclareMathOperator{\ldimh}{\underline{dim}_H}

\DeclareMathOperator\dimcor{dim_{cor}}

\newcommand{\dd}{\,\mathrm{d}}
\renewcommand{\iint}{\int\hspace{-0.1in}\int}

\newenvironment{labeledlist}[2][\unskip]
{ \renewcommand{\theenumi}{({#2}\arabic{enumi}{#1})}
  \renewcommand{\labelenumi}{({#2}\arabic{enumi}{#1})}
  \begin{enumerate} }
{ \end{enumerate} }

\begin{document}

\title{A note on correlation and local dimensions}

\author{Jiaojiao Yang}
\address[Jiaojiao Yang]{Department of Mathematics, South China University of
Technology, Guangzhou, 510641, P.~R. China}
\email{y.jiaojiao@mail.scut.edu.cn}

\author{Antti K\"aenm\"aki}
\address[Antti K\"aenm\"aki]{Department of Mathematics and Statistics \\
         P.O.\ Box 35 (MaD) \\
         FI-40014 University of Jyv\"askyl\"a \\
         Finland}
\email{antti.kaenmaki@jyu.fi}

\author{Min Wu}
\address[Min Wu]{Department of Mathematics, South China University of
Technology, Guangzhou, 510641, P.~R. China}
\email{wumin@scut.edu.cn}

\thanks{MW is the corresponding author. The authors are listed in
non-alphabetical order only because the regulations of the SCUT require that the PhD student is the first-named author. This work is supported by the NSFC (Grant
No.~11371148). AK is indebted to the mobility grant of the Academy of Finland and the NSFC (Grant No.~11411130372), and also thanks the SCUT, where this project was initiated, for warm hospitality.}
\subjclass[2000]{Primary 28A80; Secondary 30C65}
\keywords{correlation dimension, local dimension, Moran construction, finite
clustering property}
\date{\today}

\begin{abstract}
  Under very mild assumptions, we give formulas for the correlation and local dimensions of measures on the limit set of a Moran construction by means of the data used to construct the set.
\end{abstract}

\maketitle

\section{Introduction}
The correlation dimension was introduced in \cite{PGH}. It is widely used in numerical investigations of dynamical systems. The properties of the correlation dimension and the $L^q$-spectrum has been studied for various types of attractors of iterated function systems; for example, see \cite{CHY, MMR, P97, PT95, SS98}. We continue this line of research and in Proposition \ref{result 5}, we complement the study initiated in \cite{KLS, KRS13}.

The other important object in this note is the local
dimension of a measure. It has a close connection with the theory of
Hausdorff and packing dimensions of a set. Therefore it is a
classical problem to try to express the local dimension by means
of the data used to construct the set; for example, see
\cite{BK, CM92, KLS, KR15, LW11, S90}. Our main result in this note is
Theorem \ref{result 8}. Under a natural separation condition, the finite clustering property, it solves this problem completely.

\section{Correlation dimension via general filtrations} \label{sec:correlation}

Let $(X, d)$ be a compact metric space and $\mu$ a locally finite
Borel regular measure supported in $X$. Since the metric will always be clear from the content, we simply denote $(X,d)$ by $X$. Recall that the support of a measure $\mu$, denoted by $\spt(\mu)$, is the smallest closed subset of $X$ with full $\mu$-measure. For $s\geq 0$ and $x\in X$, define the \emph{$s$-potential} of $\mu$ at the point $x$ to be
\begin{equation*}
  \phi_{s}(x)= \int d(x, y)^{-s} \dd\mu(y),
\end{equation*}
where $d(x, y)$ is the distance between two points $x$ and $y$ in $X$.
Furthermore, define the \emph{$s$-energy} of $\mu$ to be
\begin{equation*}
  I_{s}(\mu)=\int \phi_{s}(x) \dd\mu(x)=\iint d(x, y)^{-s}\dd\mu(x) \dd\mu(y).
\end{equation*}
For the basic properties of the $s$-energy, the reader is referred to the
Mattila's book \cite[\S 8]{M95}. The quantity
\begin{equation*}
  \dimcor(\mu) = \inf \{s : I_{s}(\mu)= \infty\} = \sup \{ s : I_s(\mu) < \infty \}
\end{equation*}
is called the \emph{correlation dimension} of the measure $\mu$. Measure-theoretical
properties of this dimension map are studied in \cite{MMR}.

We now recall the definition of the local dimension of measures. Let
$\mu$ be a locally finite Borel regular measure on metric space $X$.
The \emph{lower} and \emph{upper local dimensions} of the measure
$\mu$ at a point $x \in X$ are defined respectively by
\begin{align*}
  \ldimloc(\mu,x)&=\liminf_{r \downarrow 0}\frac{\log \mu(B(x,r))}{\log r}, \\
  \udimloc(\mu,x)&=\limsup_{r \downarrow 0}\frac{\log \mu(B(x,r))}{\log r}.
\end{align*}
Here $B(x,r)$ is the closed ball of radius $r>0$ centered at $x\in X$.
We also define the \emph{lower Hausdorff dimension} of the measure
$\mu$ by setting
\begin{equation*}
  \ldimh(\mu)=\essinf_{x\sim \mu} \ldimloc(\mu,x).
\end{equation*}
The correlation dimension of a measure $\mu$ is at most
the lower Hausdorff dimension of the measure $\mu$. We recall the proof of
this simple fact in the following lemma.

\begin{lemma} \label{correlation_lemmas}
  If $X$ is a compact metric space and $\mu$ is a
  finite Borel regular measure on $X$, then
  \begin{equation*}
    \ldimloc(\mu,x) = \inf\{ s : \phi_s(x)=\infty \} = \sup\{ s : \phi_s(x) < \infty \}
  \end{equation*}
  for all $x \in X$. Furthermore,
  \begin{equation*}
    \dimcor(\mu) = \liminf_{r\downarrow 0}\frac{\log\int\mu(B(x,r))\dd\mu(x)}{\log r} \leq \ldimh(\mu).
  \end{equation*}
\end{lemma}

\begin{proof}
  Fix $x \in X$. If $s$ is so that $\phi_s(x)<\infty$, then
  \begin{equation} \label{eq:int_phi}
    r^{-s}\mu(B(x,r))\leq \int_{B(x,r)}d(x, y)^{-s} \dd\mu(y) \leq \phi_s(x)
    \quad \text{for all } r>0.
  \end{equation}
  It follows that $\ldimloc(\mu,x) \geq s$ and thus,
  \begin{equation*}
    \ldimloc(\mu,x) \geq \sup\{ s : \phi_s(x)<\infty \}.
  \end{equation*}
  To show that $\dimcor(\mu) \le \ldimh(\mu)$, fix $s > \ldimh(\mu)$. Notice that there exists
  a set $A$ with $\mu(A)>0$ such that $\ldimloc(\mu,x)< s$ for all $x\in A$. The above reasoning implies that $\phi_s(x)=\infty$ for all $x \in A$. Therefore $I_s(\mu)=\infty$ and the claim follows.
  Similarly, if $s < \dimcor(\mu)$, then, by integrating \eqref{eq:int_phi}, we see that
  \begin{equation*}
    r^{-s}\int\mu(B(x,r))\dd\mu(x)\leq I_{s}(\mu)<\infty \quad \text{for all } r>0.
  \end{equation*}
  Therefore
  \begin{equation} \label{eq:int_cor}
    \liminf_{r\downarrow 0}\frac{\log\int\mu(B(x,r))\dd\mu(x)}{\log r}\geq \dimcor(\mu).
  \end{equation}
  To show the remaining inequalities, fix $t<s<\ldimloc(\mu,x)$. Observe that now there exists
  $r_0>0$ such that $\mu(B(x,r))<r^s$ for all $0<r<r_0$. Thus
  \begin{align*}
    \phi_t(x) &= \int d(x, y)^{-t}\dd\mu(y)=t\int_{0}^{\infty}r^{-t-1} \mu(B(x,r))\dd r\\& \leq
    t\int_{0}^{r_{0}}r^{s-t-1}\dd r + t\int_{r_{0}}^{\infty}r^{-t-1}
    \mu(B(x,r))\dd r < \infty
  \end{align*}
  and $\inf\{ s : \phi_s(x)=\infty \} \ge t$. The proof of the converse inequality of \eqref{eq:int_cor} is similar and thus omitted.
\end{proof}

\begin{remark} \label{correlation_remark}
  (1) If there exist $A \subset X$ and $s,r_0,c>0$ such that $\mu(B(x,r)) \le cr^s$ for all $0<r<r_0$ and $x \in A$, then Lemma \ref{correlation_lemmas} implies that $\dimcor(\mu|_A) \ge s$. In particular, if $\mu$ is a finite measure, then for every $\eps > 0$ there exists a compact set $A$ with $\mu(X \setminus A)<\eps$ such that $\dimcor(\mu|_A) \ge \ldimh(\mu)$. To see this, fix $\eps>0$ and let $(s_i)_{i \in \N}$ be a strictly increasing sequence converging to $\ldimh(\mu)$. Egorov's theorem implies that for every $i$ there are $r_i>0$ and a compact set $A_i \subset X$ with $\mu(X \setminus A_i) < 2^{-i}\eps$ such that $\mu(B(x,r)) < r^{s_i}$ for all $0<r<r_i$ and $x \in A_i$. Defining $A = \bigcap_{i=1}^\infty A_i$ we have $\mu(X \setminus A) \le \sum_{i=1}^\infty \mu(X \setminus A_i) < \eps$. Fix $N \in \N$ and let $B_N = \bigcap_{i=1}^N A_i$. Then $\mu(B(x,r)) < r^{s_N}$ for all $0<r<\min_{i\in\{1,\ldots,N\}}r_i$ and $x \in B_N \supset A$. This gives $\dimcor(\mu|_A) \ge s_N$ and, as $N$ was arbitrary, finishes the proof.
  
  (2) Let us consider the standard $\tfrac13$-Cantor set and define $\mu_p$ to be the Bernoulli measure associated to the probability vector $(p,1-p)$. It is well known that $\dimcor(\mu_p) = -\log_3(p^2 + (1-p)^2)$; for example, see Proposition \ref{result 5}. Recalling e.g.\ \cite[Proposition 10.4]{F97}, we see that $\dimcor(\mu_p) < \ldimh(\mu_p)$ for all $p \in (0,1) \setminus \{ 1/2 \}$.
\end{remark}

If the metric space $X$ is doubling, then we
can define the correlation dimension via a discrete process.
More precisely, we will see that the  definition can be given in terms of
general filtrations. These filtrations can be considered to be generalized dyadic cubes. This gives a way to calculate the correlation dimension in many Moran constructions; see Corollary \ref{thm:moran-cor}.

Before stating the theorem,
we recall the definitions of the doubling metric space and the
general filtration. A metric space $X$ is said to be \emph{doubling}, if there is a doubling constant
$N=N(X)\in \mathbb{N}$ such that any
closed ball $B(x,r)$ with center $x\in X$
and radius $r> 0$ can be covered by $N$ balls of radius $r/2$. A
doubling metric space is always separable and the doubling property
can be stated in several equivalent ways. For instance, a metric
space $X$ is doubling if and only if there are $0<s,C< \infty$
such that if $\mathcal {B}$ is an $r$-packing of a closed ball
$B(x,R)$ with $0<r<R$, then the cardinality of $\mathcal {B}$ is at
most $C(R/r)^{s}$. Here the \emph{$r$-packing} $\mathcal {B}$ of a set $A$
is a collection of disjoint closed balls having radius $r$.
We write $\lambda B(x, r)= B(x, \lambda r)$ for $\lambda \in (0,\infty)$.

The definition of the general filtration is introduced in \cite{KLS}.
We assume that $(\delta_{n})_{n\in \mathbb{N}}$ and
$(\gamma_{n})_{n\in \mathbb{N}}$ are two decreasing sequences of
positive real numbers satisfying
\begin{labeledlist}{F}
  \item $\delta_n \le \gamma_n$ for all $n \in \N$, \label{F-smaller}
  \item $\lim_{n \to \infty} \gamma_n = 0$, \label{F-gammalim}
  \item $\lim_{n \to \infty} \log\delta_n/\log\delta_{n+1} = 1$, \label{F-deltalim}
  \item $\lim_{n \to \infty} \log\gamma_n/\log\delta_n = 1$. \label{F-gammadeltalim}
\end{labeledlist}
For each $n\in \mathbb{N}$, let $\mathcal {Q}_{n}$ be a
collection of disjoint Borel subsets of the doubling metric space
$X$ such that each $Q\in \mathcal {Q}_{n}$ contains a ball $B_{Q}$
of radius $\delta_{n}$ and is contained in a ball $B^{Q}$ of radius
$\gamma_{n}$. Define
\begin{equation*}
E=\bigcap\limits_{n\in \mathbb{N}}\bigcup\limits_{Q\in \mathcal
{Q}_{n}}Q.
\end{equation*}
The collection $\{\mathcal {Q}_{n}\}_{n\in \mathbb{N}}$ is called the \emph{general
filtration} of $E$.

The classical dyadic cubes of the Euclidean space is an example of a general filtration.
Such kind of nested constructions can also be defined on doubling metric spaces (see e.g.\ \cite{KRS12b}) and
these constructions also serve as examples. It should be noted that the nested structure is not always a necessity; consult \cite{KRS12} for such examples. In Lemma \ref{moran_filtration}, we show that certain Moran constructions are general filtrations. These constructions include, for example, all the self-conformal sets satisfying the strong separation condition.

Besides giving the desired discrete version of the definition, the following result states also that the correlation dimension is in fact the $L^2$-spectrum of the measure; see \cite[Proposition 3.2]{KLS}. In $\R^n$, this is of course well known -- the point here is that the equivalence now covers the general filtrations case too.

\begin{proposition}\label{result 5}
  Let $X$ be a compact doubling metric space. If $\{\mathcal{Q}_n\}_{n \in \N}$ is a general filtration of $E$ and $\mu$ is a
  finite Borel regular measure on $E$, then
  \begin{equation*}
    \dimcor(\mu)=\liminf_{n\to\infty} \frac{\log\sum_{Q\in
  \mathcal {Q}_{n}} \mu(Q)^{2}}{\log \delta_{n}}.
  \end{equation*}
\end{proposition}

\begin{proof}
Observe that for each $Q\in \mathcal {Q}_{n}$ we have $Q\subset B(x, 2\gamma_{n})$
for all $x\in Q$. Therefore
\begin{equation*}
  \int\mu(B(x,2\gamma_n))\dd\mu(x)=\sum_{Q\in \mathcal
  {Q}_{n}}\int_{Q}\mu(B(x,2\gamma_n))\dd\mu(x)\geq \sum_{Q\in \mathcal {Q}_{n}}
  \mu(Q)^{2}
\end{equation*}
and it follows from Lemma \ref{correlation_lemmas} and \ref{F-gammadeltalim} that
\begin{equation*}
  \dimcor(\mu)\leq \liminf_{n \to \infty} \frac{\log\int\mu(B(x,2\gamma_n))\dd\mu(x)}{\log 2\gamma_n}
  \leq \liminf_{n\rightarrow\infty} \frac{\log\sum_{Q\in
\mathcal {Q}_{n}} \mu(Q)^{2}}{\log \delta_{n}}.
\end{equation*}
To show the other inequality, fix $r>0$ and let $n \in \N$ be such that $\gamma_{n+1} \le r < \gamma_n$.
Choose for each $Q\in \mathcal {Q}_{n}$ balls
$B_{Q}$ of radius $\delta_{n}$ and $B^Q$ of radius $\gamma_n$ so that $B_{Q}\subset Q \subset B^Q$. Now for each $Q\in
\mathcal {Q}_{n}$ we have
\begin{equation*}
  Q \subset B(x, 2\gamma_{n})\subset B_{Q}[4\gamma_{n}] \subset
  \bigcup_{Q'\in \mathcal {C}_{Q}} Q'
\end{equation*}
for all $x\in Q$, where $\mathcal {C}_{Q}=\{Q'\in \mathcal {Q}_{n} : Q'\cap
B_{Q}[4\gamma_{n}] \neq \emptyset \}$ and $B[r]$ denotes the ball of
radius $r$ having the same center as the ball $B$. Thus
\begin{equation} \label{eq:r5_estimate1}
  \int\mu(B(x,2\gamma_n))\dd\mu(x)=\sum_{Q\in \mathcal
  {Q}_{n}}\int_{Q}\mu(B(x,2\gamma_n))\dd\mu(x) \leq \sum_{Q\in
  \mathcal {Q}_{n}}\mu(Q) \sum_{Q'\in \mathcal {C}_{Q}}\mu(Q').
\end{equation}
Observe that if $Q \in \QQ_n$, then
\begin{equation} \label{eq:r5_estimate2}
  \mu(Q) \sum_{Q' \in \CC_Q} \mu(Q') \leq \biggl( \sum_{Q' \in \CC_Q} \mu(Q') \biggr)^2 \leq \#\CC_Q \sum_{Q' \in \CC_Q} \mu(Q')^2.
\end{equation}
Here $\# \mathcal {C}_{Q}$ is the cardinality of the set $\mathcal {C}_{Q}$.
Let us next estimate this cardinality.

Fix $Q \in \QQ_n$ and $Q' \in \CC_Q$. Since $Q'\cap
B_{Q}[4\gamma_{n}]\neq \emptyset$ we have $B^{Q'} \cap B_{Q}[4\gamma_{n}] \neq \emptyset$ and
\begin{equation*}
  B_{Q'} \subset Q' \subset B^{Q'} \subset B_{Q}[6\gamma_{n}].
\end{equation*}
Thus the collection $\{B_{Q'}: Q'\in \mathcal {C}_{Q}\}$
is a $\delta_{n}$-packing of the ball $B_{Q}[6\gamma_{n}]$.
It follows from the doubling property of $X$ that
\begin{equation*}
  \#\CC_Q = \#\{B_{Q'}: Q'\in \mathcal {C}_{Q}\}\leq
  C\Bigl(\frac{6\gamma_{n}}{\delta_{n}}\Bigr)^{s}.
\end{equation*}
Recalling \eqref{eq:r5_estimate1} and \eqref{eq:r5_estimate2}, this estimate gives
\begin{equation} \label{eq:r5_estimate3}
  \int\mu(B(x,2\gamma_n))\dd\mu(x) \leq C\Bigl(\frac{6\gamma_{n}}{\delta_{n}}\Bigr)^{s} \sum_{Q \in \QQ_n} \sum_{Q' \in \CC_Q} \mu(Q')^2.
\end{equation}
To get rid of the double-sum, we need to estimate the cardinality of the set
$\DD_{Q'} = \{ Q \in \QQ_n : Q' \in \CC_Q \}$. If $Q \in \QQ_n$ is such that
$Q' \in \CC_Q$, then $B^{Q'} \cap B_Q[4\gamma_n] \ne \emptyset$ and $B_Q \subset B^{Q}[4\gamma_n] \subset B^{Q'}[9\gamma_n]$.
Thus the collection $\{ B_Q : Q \in \DD_{Q'} \}$ is a $\delta_n$-packing of the ball $B^{Q'}[9\gamma_n]$ and, as above,
\begin{equation*}
  \#\DD_{Q'} = \#\{ B_Q : Q \in \DD_{Q'} \} \leq C\Bigl( \frac{9\gamma_n}{\delta_n} \Bigr)^s.
\end{equation*}
Therefore, it follows from \eqref{eq:r5_estimate3} that
\begin{equation*}
  \int\mu(B(x,2\gamma_n))\dd\mu(x) \leq C^2\Bigl(\frac{6\gamma_{n}}{\delta_{n}}\Bigr)^{s} \Bigl( \frac{9\gamma_n}{\delta_n} \Bigr)^s \sum_{Q \in \QQ_n} \mu(Q)^2.
\end{equation*}
This, together with \ref{F-deltalim} and \ref{F-gammadeltalim}, yields
\begin{align*}
  \dimcor(\mu) &\geq \liminf_{n \to \infty} \frac{\log\int\mu(B(x,2\gamma_n))\dd\mu(x)}{\log\gamma_{n+1}}
  \geq \liminf_{n \to \infty} \frac{\log \sum_{Q \in \QQ_n} \mu(Q)^2}{\log\delta_{n}}
\end{align*}
finishing the proof.
\end{proof}

\section{Dimensions in Moran constructions}

We adopt the following notation for the rest of the
paper. For each $j \in \mathbb{N}$ we fix a finite set
$I_{j}$ and define $\Sigma = \prod_{j=1}^\infty I_{j}$. Thus, if
$\sigma \in \Sigma$, then $\sigma = i_{1}i_{2}\cdots$ where $i_{j}
\in I_{j}$ for all $j \in \mathbb{N}$. With the product discrete topology, $\Sigma$ is compact.
We let $\sigma|_{n} = i_{1} \cdots i_{n}$ and $\Sigma_{n} = \{
\sigma|_{n} : \sigma \in \Sigma \}$ for all $n \in \mathbb{N}$. We
also define $\Sigma_{*} = \bigcup_{n=1}^\infty \Sigma_{n}$. The
concatenation of $\sigma \in \Sigma_{*}$ and $\delta \in \Sigma_{*}
\cup \Sigma$ is denoted by $\sigma\delta$. The length of $\sigma \in
\Sigma_{*}$ is denoted by $|\sigma|$. For $\sigma \in \Sigma_{*}$ we
set $\sigma^- = \sigma|_{|\sigma|-1}$ and $[\sigma] = \{
\omega \in \Sigma : \omega|_{|\sigma|} = \sigma \}$. Finally, we define $\sigma|_{0}$ to be $\varnothing$.

Let $X$ be a complete separable metric space. A \emph{Moran
construction} is a collection $\{ E_\sigma : \sigma \in \Sigma_{*}
\}$ of compact subsets of $E_\varnothing \subset X$ satisfying the following two
properties:
\begin{labeledlist}{M}
  \item $E_{\sigma} \subset E_{\sigma^{-}}$ for all $\sigma \in
  \Sigma_{*}$, \label{M1}
  \item $\diam(E_{\sigma|_{n}}) \to 0$ as $n \to \infty$ for all
  $\sigma \in \Sigma$. \label{M2}
\end{labeledlist}
Here $\diam(A)$ denotes the diameter of the set $A$.
The \emph{limit set} of the Moran construction is the compact set
\begin{equation*}
  E = \bigcap \limits_{n\in \N}\bigcup\limits_{\sigma\in \Sigma_{n}}
  E_{\sigma}.
\end{equation*}
Each $E_\sigma$ is
called a \emph{construction set} of level $|\sigma|$. We emphasize that
the placement of the construction sets at each level of the Moran
construction can be arbitrary, and they need not be disjoint.
Clearly, self-similar sets introduced in \cite{H81}, (finitely generated) self-conformal sets studied in \cite{MU96}, and Moran
sets defined in \cite{CM92, KLS, KR15, KV08, LW11, RV11, W00} are all special
cases of the limit sets of Moran constructions.
If $\sigma\in \Sigma$, then
the single point in the set $\bigcap_{n\in \N}E_{\sigma|_{n}}\subset E$ is
denoted by $\pi(\sigma)$. This defines a surjective mapping $\pi \colon \Sigma \to E$
which is continuous in the product discrete topology.

The following lemma connects Moran constructions to general filtrations considered in \S \ref{sec:correlation}.

\begin{lemma} \label{moran_filtration}
  Let $\{ E_\sigma : \sigma \in \Sigma_{*} \}$ be a Moran construction on a complete separable metric space satisfying \ref{M1}, \ref{M2},
  \begin{labeledlist}{M}
    \setcounter{enumi}{2}
    \item $E_{\sigma i} \cap E_{\sigma j} = \emptyset$ for all $\sigma i, \sigma j \in \Sigma_*$ with $i \ne j$, \label{M3}
    \item there exists $C_0>0$ such that for each $\sigma\in\Sigma_*$ there is $x\in E_\sigma$ for which $$B\left(x,C_0 \diam(E_\sigma)\right)\subset E_\sigma,$$ \label{M4}
    \item it holds that $$\lim_{n \to \infty} \frac{\log\diam(E_{\sigma|_n})}{\log\min\{ \diam(E_{\omega}) : \omega \in \Sigma_{n+1} \text{ such that } \omega^- = \sigma|_n \}}=1,$$ where the convergence is uniform for all $\sigma \in \Sigma$. \label{M5}
  \end{labeledlist}
  If $\QQ_n = \{ E_\sigma : \diam(E_\sigma) \le \gamma_n < \diam(E_{\sigma^-}) \}$, where $\gamma_n = \min\{ \diam(E_\sigma) : \sigma \in \Sigma_n \}$, then there exists a sequence $(\delta_n)_{n \in \N}$ such that $\{ \QQ_n \}_{n \in \N}$ is a general filtration for $E$.
\end{lemma}

\begin{proof}
  The claim is proved in the course of the proof of \cite[Lemma 4.2]{KLS}.
\end{proof}

A self-conformal set satisfying the strong separation condition is a simple example of a limit set of a Moran construction satisfying the assumptions of Lemma \ref{moran_filtration}. This can be easily seen since the condition
\begin{labeledlist}{M}
  \setcounter{enumi}{5}
  \item there exists $0<c\le 1$ such that $\diam(E_\sigma) \ge c\diam(E_{\sigma^-})$ for all $\sigma \in \Sigma_* \setminus \{ \varnothing \}$, \label{M6}
\end{labeledlist}
which is satisfied by any self-conformal set, together with \ref{M1} implies \ref{M5}.

Recall that in complete separable metric spaces locally finite Borel regular measures are Radon. %; see \cite[Part I, \S II.3]{S73} or \cite[Theorem V.5.3]{J78}.
If $X$ and $Y$ are complete separable spaces and $\mu$ is a Radon measure with compact support on $Y$, then the \emph{pushforward measure} of $\mu$ under a continuous mapping $f \colon Y \to X$ is denoted by $f\mu$. In this case, $f\mu$ is a Radon measure and $\spt(f\mu) = f(\spt(\mu))$. %; see \cite[Theorem 1.18]{M95}.

\begin{corollary} \label{thm:moran-cor}
  Let $\{ E_\sigma : \sigma \in \Sigma_* \}$ be a Moran construction on a compact doubling metric space satisfying \ref{M1}--\ref{M4} and
  \begin{labeledlist}{M}
    \setcounter{enumi}{6}
    \item there exist $C\ge 1$ and a sequence $(\beta_n)_{n \in \N}$ with $\log\beta_n/\log\beta_{n+1} \to 1$ such that $C^{-1}\beta_n \le \diam(E_\sigma) \le C\beta_n$ for all $\sigma \in \Sigma_n$ and $n \in \N$. \label{M7}
  \end{labeledlist}
  If $\mu$ is a finite Borel regular measure on $\Sigma$, then
  \begin{equation*}
    \dimcor(\pi\mu) = \liminf_{n \to \infty} \frac{\log\sum_{\omega \in \Sigma_n} \mu([\omega])^2}{\log\beta_n}.
  \end{equation*}
\end{corollary}

\begin{proof}
  Observe that \ref{M7} clearly implies \ref{M5}. Furthermore, it also guarantees that $C^{-1}\beta_n \le \gamma_n \le C\beta_n$ for all $n \in \N$, where $\gamma_n$ is as in Lemma \ref{moran_filtration}. Therefore, the claim follows immediately from Lemma \ref{moran_filtration} and Proposition \ref{result 5}.
\end{proof}

We will next start studying local dimensions of a measure. For that we need the following separation condition. Given a Moran construction $\{ E_\sigma : \sigma \in \Sigma_{*} \}$, we set
\begin{equation*}
  N(x,r) = \{ \sigma \in \Sigma_{\ast} : \diam(E_\sigma) \leq r <
  \diam(E_{\sigma^{-}})
  \text{ and } E_\sigma \cap B(x,r) \ne \emptyset \}
\end{equation*}
for all $x\in E$ and $r> 0$.
We say that the Moran construction satisfies the \emph{finite
clustering property} if
\begin{equation*}
  \sup_{x \in E} \limsup_{r \downarrow 0} \# N(x,r) < \infty.
\end{equation*}
Observe that a Moran construction on a compact doubling metric space satisfying \ref{M3}, \ref{M4}, and \ref{M6} satisfies the finite clustering property. For a self-conformal set, the finite clustering property is equivalent to the open set condition; see e.g.\ \cite[Corollary 5.8]{KV08}. For more examples, the reader is referred to \cite{KR15}. Also, Moran sets considered in \cite{W00} (with $c_* > 0$) satisfy the finite clustering property.

The following proposition plays an important role in the proof of Theorem
\ref{result 8}.

\begin{proposition}\label{result 7}
Suppose that $X$ and $Y$ are complete separable metric spaces and
$f\colon Y \to X$ is continuous. If $\mu$ is a locally finite
Borel measure with compact support on $Y$ and  $A \subset Y$ is
such that $\mu(A)>0$, then
\begin{align*}
  \ldimloc(f\mu,f(y)) &= \ldimloc(f(\mu|_A),f(y)), \\
  \udimloc(f\mu,f(y)) &= \udimloc(f(\mu|_A),f(y)),
\end{align*}
for $\mu$-almost all $y \in A$.
\end{proposition}

\begin{proof}
  Clearly always $\ldimloc(f\mu,f(y)) \le \ldimloc(f(\mu|_A),f(y))$. Let us assume that there are a bounded set $A' \subset A$ and $\gamma>0$ such that
  \begin{equation*}
    \ldimloc(f\mu,f(y)) + \gamma < \ldimloc(f(\mu|_A),f(y))
  \end{equation*}
  for all $y \in A'$. Then for every $y \in A'$ there is a sequence $r_j \downarrow 0$ such that
  \begin{equation*}
    \mu(A \cap f^{-1}(B(f(y),5r_j))) < (5r_j)^\gamma f\mu(B(f(y),r_j))
  \end{equation*}
  for all $j \in \N$. Since $f\mu$ is a Radon measure we find an open set $U \subset X$ with $U \supset f(A')$ and $f\mu(U)<\infty$. Let $\roo>0$ and, by relying on the $5r$-covering theorem (see e.g.\ \cite[Theorem 1.2]{H01}), choose a countable disjoint subcollection $\BB_\roo$ of
  \begin{align*}
    \{ B(f(y),r) : \;&y \in A' \text{ and } 0<r<\roo \text{ such that} \\ &B(f(y),r) \subset U \text{ and } \mu(A \cap f^{-1}(B(f(y),5r))) < (5r)^\gamma f\mu(B(f(y),r)) \}
  \end{align*}
  such that $5\BB_\roo$ covers $f(A')$. Since now
  \begin{equation*}
    \mu(A') \le \sum_{B \in \BB_\roo} \mu(A \cap f^{-1}(5B)) < \sum_{B \in \BB_\roo} (5\roo)^\gamma f\mu(B) \le (5\roo)^\gamma f\mu(U)
  \end{equation*}
  the claim follows by letting $\roo \downarrow 0$. The proof for the upper local dimension is similar and thus omitted.
\end{proof}

\begin{remark}
  We say that a measure $\mu$ on $X$ has the \emph{density point property} if
  \begin{equation*}
    \lim_{r \downarrow 0} \frac{\mu(A \cap B(x,r))}{\mu(B(x,r))} = 1
  \end{equation*}
  for $\mu$-almost all $x \in A$ whenever $A \subset X$ is $\mu$-measurable. It was demonstrated in \cite[Example 5.6]{KRS12} that the density point property is not necessarily valid for all measures even if the space $X$ is doubling. Despite of this, Proposition \ref{result 7} implies that
  \begin{align*}
    \ldimloc(\mu,x) &= \ldimloc(\mu|_A,x), \\
    \udimloc(\mu,x) &= \udimloc(\mu|_A,x),
  \end{align*}
  for $\mu$-almost all $x \in A$.
\end{remark}

We will next show that under the finite clustering property, the lower local
dimension of the pushforward measure can be obtained symbolically. This result
generalizes \cite[Lemma 4.2]{KLS} and \cite[Proposition 3.10]{KR15}. The proof of \cite[Lemma 4.2]{KLS} used \ref{M5} and assumed that the construction sets of the same level were disjoint and had a certain shape. In \cite[Proposition 3.10]{KR15}, it was assumed that the Moran construction is so rigid that
\begin{labeledlist}{M}
  \setcounter{enumi}{7}
  \item there exists $D \ge 1$ such that $\diam(E_{\iii\jjj}) \le D\diam(E_\iii) \diam(E_\jjj)$ for all $\iii\jjj \in \Sigma_∗$. \label{M8}
\end{labeledlist}
Although constructions given by iterated function systems satisfy \ref{M8}, it rules out many interesting Moran constructions. The proof also used uniform version of the finite clustering property. In the main part of the following theorem, we do not require any of these assumptions.

\begin{thm}\label{result 8}
  If $\{ E_\sigma : \sigma \in \Sigma_* \}$ is a Moran construction on a complete separable metric space satisfying \ref{M1}, \ref{M2}, and the finite clustering property, and $\mu$ is a finite Borel regular measure on $\Sigma$, then
  \begin{align}
    \ldimloc(\pi\mu,\pi(\sigma)) &= \liminf_{n \to
    \infty} \frac{\log \mu([\sigma|_n])}{\log \diam(E_{\sigma|_n})}, \label{eq:localdim} \\
    \udimloc(\pi\mu,\pi(\sigma)) &\geq \limsup_{n \to
    \infty} \frac{\log \mu([\sigma|_n])}{\log \diam(E_{\sigma|_n})}, \label{eq:localdim2}
  \end{align}
  for $\mu$-almost all $\sigma \in \Sigma$. Furthermore, if the Moran construction also satisfies \ref{M5}, then \eqref{eq:localdim2} holds with an equality.
\end{thm}

\begin{proof}
We may clearly assume that $\mu$ has no atoms.
Since $E_{\sigma|_n} \subset
B(\pi(\sigma),\diam(E_{\sigma|_n}))$ and $\mu([\sigma|_n])
\le \pi\mu(E_{\sigma|_n})$ for all $\sigma \in \Sigma$ and $n \in
\mathbb{N}$, the right-hand side of \eqref{eq:localdim} is an upper
bound for the lower local dimension. Furthermore, if $\eps > 0$, then \ref{M5} implies that there exists $n_0 \in \N$ such that $\diam(E_{\sigma|_n}) < \diam(E_{\sigma|_{n+1}})^{1-\eps}$. Thus
\begin{equation*}
  \frac{\log \pi\mu(B(\pi(\sigma),r))}{\log r} \le \frac{\log \pi\mu(B(\pi(\sigma),\diam(E_{\sigma|_{n+1}})))}{\log \diam(E_{\sigma|_n})} \le \frac{\log \mu([\sigma|_{n+1}])}{(1-\eps)\log \diam(E_{\sigma|_{n+1}})}
\end{equation*}
for all $r>0$ and $n \ge n_0$ with $\diam(E_{\sigma|_{n+1}}) < r \le \diam(E_{\sigma|_n})$. This shows that, assuming \ref{M5},  the right-hand side of \eqref{eq:localdim2} is an upper bound for the upper local dimension.

To show that the right-hand side of \eqref{eq:localdim} is also an lower bound, suppose to the contrary that there is a set $A' \subset E$
with $\mu(A')>0$ and $s>0$ such that
\begin{equation*}
  \ldimloc(\pi\mu,\pi(\sigma)) < s < \liminf_{n \to \infty} \frac{\log \mu([\sigma|_n])}{\log \diam(E_{\sigma|_n})}
\end{equation*}
for all $\sigma \in A'$. By Egorov's theorem, the limit above is
uniform in a set $A \subset A'$ with $\mu(A)>0$: there is $n_0 \in
\mathbb{N}$ such that
\begin{equation} \label{eq:estimate1}
  \mu([\omega|_n]) < \diam(E_{\omega|_n})^s
\end{equation}
for all $\omega \in A$ and $n \ge n_0$.

Let $\sigma \in A$ and, relying on the finite clustering property,
choose $M \in \mathbb{N}$ and $r_0>0$ such that $\# N(\pi(\sigma),r)
\le M$ for all $0<r<r_0$. We trivially have
\begin{equation*}
  \pi(\mu|_A)(B(\pi(\sigma),r)) \le \sum_{\omega \in
  N(\pi(\sigma),r)} \mu|_A([\omega])
\end{equation*}
for all $0<r<r_0$.
Observe that if $\omega \in \Sigma_*$ satisfies $\diam(E_\omega) <
\min\{ \diam(E_\tau): \tau \in \Sigma_{n_0} \}$, then $|\omega|
> n_0$. Therefore, assuming $r_0 < \min\{\diam(E_\tau) : \tau \in
\Sigma_{n_0} \}$, we may apply the estimate \eqref{eq:estimate1} for
each $\omega \in N(\pi(\sigma),r)$ whenever $[\omega] \cap A \ne
\emptyset$ and $0<r<r_0$. Thus
\begin{equation*}
  \pi(\mu|_A)(B(\pi(\sigma),r)) \le \sum_{\omega \in
  N(\pi(\sigma),r)} \diam(E_\omega)^s \le Mr^s
\end{equation*}
and, consequently,
$\ldimloc(\pi(\mu|_A),\pi(\sigma)) \ge s >
\ldimloc(\pi\mu,\pi(\sigma))$ for all $\sigma
\in A$. This contradicts with Proposition \ref{result 7}.
The proof of \eqref{eq:localdim2} is similar and thus omitted.
\end{proof}

\begin{remark}
  It would be interesting to know if the inequality in \eqref{eq:localdim2} can be strict under the assumptions of Theorem \ref{result 8}. We remark that a non-uniform version of \ref{M5} suffices for the equality.
\end{remark}

\begin{example}
  For each $j \in \N$ let $I_j$ be a finite set, $N_j = \# I_j$, and $\Sigma = \prod_{j=1}^\infty I_j$. Let $p = (p_{j1},\ldots,p_{jN_j})$ be a positive probability vector for all $j \in \N$ and $\mu$ the probability measure for which $\mu([\sigma]) = \prod_{j=1}^n p_{ji_j}$ for all $\sigma = i_1\cdots i_n \in \Sigma_n$ and $n \in \N$. We assume that $\{ E_\sigma : \sigma \in \Sigma_* \}$ is a Moran construction on a compact doubling metric space satisfying \ref{M3}, \ref{M4}, \ref{M7}, and the finite clustering property. It follows from Theorem \ref{result 8} that
  \begin{equation*}
    \ldimloc(\pi\mu,\pi(\sigma)) = \liminf_{n \to \infty} \frac{\sum_{j=1}^n \log p_{ji_j}}{\log\beta_n}
  \end{equation*}
  for $\mu$-almost all $\sigma = i_1i_2\cdots \in \Sigma$. On the other hand, since $\sum_{\sigma \in \Sigma_n} \mu([\sigma])^2 = \prod_{j=1}^n \sum_{i=1}^{N_j} p_{ji}^2$ for all $n \in \N$, Corollary \ref{thm:moran-cor} implies that
  \begin{equation*}
    \dimcor(\pi\mu) = \liminf_{n \to \infty} \frac{\sum_{j=1}^n \log\sum_{i=1}^{N_j} p_{ji}^2}{\log\beta_n}.
  \end{equation*}
  In particular, if the measure $\mu$ is the uniform distribution, that is, $p_{ji} = N_j^{-1}$ for all $i \in I_j$ and $j \in \N$, then $\dimcor(\pi\mu) = \ldimh(\pi\mu)$.
\end{example}

\end{document}